\documentclass[12pt, reqno]{amsart}
\usepackage{amsmath, amsthm, amscd, amsfonts, amssymb, graphicx, color}
\usepackage[bookmarksnumbered, colorlinks, plainpages]{hyperref}
\hypersetup{colorlinks=true,linkcolor=red, anchorcolor=green, citecolor=cyan, urlcolor=red, filecolor=magenta, pdftoolbar=true}
\input{mathrsfs.sty}

\newtheorem{theorem}{Theorem}[section]
\newtheorem{lemma}[theorem]{Lemma}

\theoremstyle{definition}

\newtheorem{example}[theorem]{Example}

\theoremstyle{remark}
\newtheorem{remark}[theorem]{Remark}
\numberwithin{equation}{section}
\makeatletter
\newcommand*{\rom}[1]{\expandafter\@slowromancap\romannumeral #1@}
\makeatother
\begin{document}

\title[ A convergence condition for Newton-Raphson method ]{ A convergence condition for Newton-Raphson method  }

\author[H. Khandani]{Hassan Khandani$^1$ }

\address{$^1$Department of Mathematics, Mahabad-Branch Islamic Azad University, Mahabad, Iran.}
\email{khandani.hassan@yahoo.com,  khandani.hassan@gmail.com, }
\email{ khandani@iau-mahabad.ac.ir}
\address{}

\renewcommand{\subjclassname}{\textup{2020} Mathematics Subject Classification}


\begin{abstract}
In this paper we study the convergence of Newton-Raphson method. For this method there exists some convergence results which are practically not very useful and just guarantee the convergence of this method when the first term of this sequence is very close to the guessed root \cite{sulimayer}. Khandani et al. introduced a new iterative method to estimate the roots of real-valued functions \cite{khandani}. Using this method we introduce some simple and  easy-to-test conditions under which Newton-Raphson sequence converges to its guessed root even when the initial point is chosen very far from this  root. More clearly, for a real-valued second differentiable function $f:[a,c]\to \mathbb R$ with $f^{''}f\ge 0$ on $(a,c)$ where $c$ is the unique root of $f$ in $[a,c]$, the Newton-Raphson sequence  $f$ converges to $c$ for each $x_0\in[a,c]$ provided $f$ satisfies some other simple conditions on this interval. A similar result holds if $[a,c]$ be replaced with $[c,b]$.  Our study will enable us to predict accurately where Newton-Raphson sequence converges.
\end{abstract}
\maketitle
\section{Introduction and preliminaries}
Newton-Raphson method is one of the most important methods in estimating the roots of a real-valued function. As an advantage of this method we can refer to its speed of convergence. Indeed, when $f$ has second continuous derivative on an interval which contains the guessed root and the initial point is very close to this root, then this method  converges \cite{kendal,sulimayer}. To know how this method has developed through out the time we refer the reader to \cite{tj}. This method has its own disadvantages too, among them, we can refer to its high dependence on how  the initial point is chosen to start the iterate sequence. In an interval where there is a point $a$ with $f^{'}(a)=0$ the behaviour of this sequence is quite wired and unpredictable. In this manuscripts we provide some results which make sure where the initial point should be chosen to make sure the convergence of the Newton-Raphson sequence. First, we need the following definitions and results to start our study.

For any differentiable function $f$ on $[a,b]$ the Newton-Raphson sequence  is defined recursively as follows, where $x_0\in [a,b]$ is called the starting or initial point,
\begin{equation}\label{newtonseq}
x_{n+1}=x_n-\frac{f(x_n)}{f^{'}(x_n)}\text { for all } n\ge 0.
\end{equation}
Khandani et al. presented the following results which guarantee the convergence of an iterate sequence on a whole interval with very simple and practical  conditions \cite{khandani}. We use these results and provide some simple conditions under which the Newton-Raphson sequence converges.
\begin{lemma}\label{l1}[H. Khandani, F. Khojasteh \cite{khandani}]
Let $f$ be a continuous real-valued function on $[a,c]$ which is differentiable on $(a,c)$  with $f^{'}(x)\ge -1$ on $(a,c)$  and  $f(x)>x$ for each $x\in[a,c)$ and $c$ is the unique fixed point of $f$ in $[a,c]$. Let $x_0\in [a,c)$ and for each $n\ge 0$ define:
\begin{equation}\label{meanex}
x_{n+1}=\frac{x_n+f(x_n)}{2},
\end{equation}
then the sequence $\{x_n\}$  converges to $c$.
\end{lemma}

\begin{lemma}\label{l2} [H. Khandani, F. Khojasteh \cite{khandani}]
Let $f$ be a continuous real valued function on $[c,b]$ which is differentiable on $(c,b)$  with $f^{'}(x)\ge -1$ on $(c,b)$  and  $f(x)<x$ for each $x\in(c,b]$ and $c$ is the unique fixed point of $f$ in $[c,b]$. Let $x_0\in (c,b]$ and for each $n\ge 0$ define:
\begin{equation}
x_{n+1}=\frac{x_n+f(x_n)}{2},
\end{equation}
then the sequence $\{x_n\}$  converges to $c$.
\end{lemma}
In this manuscript we show the set of real numbers by $\mathbb R$. We denote the set  $\{0,1,2,\dots \}$  of none negative integers by $\mathbb N$. For each $a,b\in \mathbb R$ with $a<b$, $[a,b]=\{x\in \mathbb R:a\le x\le b\}$ and $(a,b)=\{x\in R:a< x< b\}$ will denote closed and open interval from $a$ to $b$. Let $f$ be a real valued function on $\mathbb R$  we show the left  and right derivative of f at $a$ by $f^{'}(a-),f^{'}(a+)$ respectively.

\newpage
\section{Main results}
All we need to provide a convergence theorem for the Newton-Raphson method are  Lemmas $\ref{l1}$ and $\ref{l2}$. Therefore, we present our results as follows.
\begin{theorem}\label{convergence1}
Suppose that $a,c\in \mathbb R$  with $a<c$,  $f:[a,c]\to \mathbb R$ is a real-valued function, $c$ is the unique root of $f$ in $[a,c]$. Also assume that $f^{''}(x),f^{'}(x)$ exist for each $x\in (a,c)$, $f(x)f^{''}(x)\ge 0$  for each $x\in (a,c)$, $f(x)f^{'}(x)< 0$  for each $x\in [a,c)$,  $f^{'}(x)\not =0$  for each $x\in (a,c)$, $f^{'}(a+)\not =0$,  $f^{'}(c-)\not =0$. For each $x_0\in [a,c]$ define:
\begin{equation}\label{newtonseq}
x_{n+1}=x_n-\frac{f(x_n)}{f^{'}(x_n)}\text { for all } n\ge 0.
\end{equation}
Then, $\{x_n\}$ converges to $c$ as $n\to \infty$.
\end{theorem}
\begin{proof}
For each $x\in [a,c]$, define $F(x)=x-\frac{2f(x)}{f^{'}(x)}$. By our assumptions $F(x)>x$ for each $x\in[a,c)$ and
\begin{equation}\label{newtonseq}
F^{'}(x)=-1+2\frac{f^{''}(x)f(x)}{(f^{'}(x))^2}\ge -1 \text{ for each } x\in(a,c).
\end{equation}
We have $\frac{F(x_n)+x_n}{2}=x_n-\frac{f(x_n)}{f^{'}(x_n)}=x_{n+1}$. Now, by Lemma $\ref{l1}$ $\{x_n\}$ converges to $c$ as $n\to \infty$ where $c$ is the fixed point of $F$. We have $c=c-\frac{f(c)}{f^{'}(c)}$ which follows that $f(c)=0.$
\end{proof}
\begin{theorem}\label{convergence2}
Suppose that $c,b\in \mathbb R$  with $c<b$,  $f:[c,b]\to \mathbb R$ is a real-valued function, $c$ is the unique root of $f$ in $[c,b]$. Also assume that $f^{''}(x),f^{'}(x)$ exist for each $x\in (c,b)$, $f(x)f^{''}(x)\ge 0$  for each $x\in (c,b)$, $f(x)f^{'}(x)>0$  for each $x\in (c,b]$,  $f^{'}(x)\not =0$  for each $x\in (c,b)$,  $f^{'}(c+)\not =0$ and  $f^{'}(b-)\not =0$. For each $x_0\in [c,b]$ define:
\begin{equation}\label{newtonseq}
x_{n+1}=x_n-\frac{f(x_n)}{f^{'}(x_n)}\text { for all } n\ge 0.
\end{equation}
Then, $\{x_n\}$ converges to $c$ as $n\to \infty$.
\end{theorem}
\begin{proof}
For each $x\in [c,b]$, define $F(x)=x-\frac{2f(x)}{f^{'}(x)}$. By our assumptions $F(x)<x$ for each $x\in(c,b]$ and
\begin{equation}\label{newtonseq}
F^{'}(x)=-1+2\frac{f^{''}(x)f(x)}{(f^{'}(x))^2}\ge -1 \text{ for each } x\in(c,b).
\end{equation}
We have $\frac{F(x_n)+x_n}{2}=x_n-\frac{f(x_n)}{f^{'}(x_n)}=x_{n+1}$. Now, by Lemma $\ref{l2}$ $\{x_n\}$ converges to $c$ as $n\to \infty$ where $c$ is the fixed point of $F$. We have $c=c-\frac{f(c)}{f^{'}(c)}$ which follows that $f(c)=0.$
\end{proof}
The above two results give us all is needed to know the behaviour of Newton-Raphson sequence and how the starting point should be chosen.
\begin{remark}In both of Theorem $\ref{convergence1}$ and  Theorem $\ref{convergence2}$ we see that $ff^{''}\ge 0$ on the related intervals. Example $\ref{positive}$ shows that this condition can not be replaced with $ff^{''}<0$. In fact, this example shows that if   $ff^{''}<0$ on an interval, then nothing can be said about the convergence of the Newton-Raphson sequence and its convergence is depended upon how this function has been defined out of this interval. Therefore, we choose the starting point for the Newton-Raphson sequence from intervals on which $ff^{''}\ge 0$. In both of these results $f^{'}\not = 0$ on the whole interval. It is easy to see that in each of these results the Newton-Raphson sequence is monotone. Taking these facts into account, the behaviour of Newton-Raphson sequence is under control and indeed converges to the guessed root.
\end{remark}
\begin{example}\label{positive}
Suppose that  $g(x)=x^2+x$ for each $x\in [- \frac{1}{2}, \frac{1}{2}]$. Define $f:[- \frac{1}{2}, \frac{1}{2}]\to \mathbb R$  by:
\begin{equation*}
f(x)=\begin{cases}
     x^{2}-x\quad &\text{if } \, 0\le x\le \frac{1}{2}\\
         x^{2}+x\quad &\text{if} \, - \frac{1}{2}\le x\le 0\\
     \end{cases}
\end{equation*}
We know that $f=g$ on the interval $[-\frac{1}{2},0]$, $c=0$ is the unique root of $f$ and $g$ in this interval and $f$ and $g$ satisfy all conditions of Theorem $\ref{convergence1}$ except that  $f^{''}(x)f(x)< 0$ and $g^{''}(x)g(x)<0$ for each $x\in (-\frac{1}{2},0)$. Denote the Newton-Raphson sequence of $f,g$ by $\{x_n\}$ and $\{y_n\}$ respectively with initial point $x_0=y_0=-\frac{1}{3}$. We see that  $x_0=-\frac{1}{3}$, $x_1= \frac{1}{3}$, $x_2= -\frac{1}{3}, \dots$ and the sequence  $\{x_n\}$ oscillating between $-\frac{1}{3}$ and  $\frac{1}{3}$ so it is not convergent. We also have $y_0=-\frac{1}{3}, y_1=\frac{1}{3}$. Notice that $\frac{1}{3}\in (0,\frac{1}{2})$ and $g$ satisfies all conditions of Theorem  $\ref{convergence2}$ on the interval $[0,\frac{1}{2}]$. Therefore, the Newton-Raphson sequence of $g$ at $\frac{1}{3}$, which is $\{y_1,y_2,y_3,\dots\}$ converges to $0$ by this theorem. So, $\{y_n\}$ is a convergent sequence while  the sequence $\{x_n\}$ is not. This shows that in Theorem $\ref{convergence1}$ nothing can be said about the convergence of the Newton-Raphson sequence  when $ f^{''}f< 0$ on $(a,c)$. This example shows that  when $ f^{''}f< 0$ on $(a,c)$ the convergence of the Newton-Raphson sequence is depended upon the definition of $f$ outside of this interval. Similarly, the same is true about Theorem $\ref{convergence2}$.
\end{example}
\begin{example}Suppose that $f(x)=x^{3}-2x+2$ for each $x\in \mathbb R$. $f$ has a root $c\in [-2,0]$. For $x=\pm \sqrt{\frac{{2}}{3}}$, $f^{'}(x)=0$ and we see that $ \sqrt{\frac{{2}}{3}}$ is the local minimum of $f$. $f(\sqrt{\frac{{2}}{3}})>0$ which follows that $c$ is the unique root of $f$. $f>0$ on $(c,+\infty)$, $f<0$ on $(-\infty,c)$. We see that $f^{''}(x)f(x)<0$ for each $x\in (c,0)$, therefore we don't choose the initial point $x_0$ in $(c,0)$. The interval $[0,+\infty)$ is also not suitable to choose the initial point from. Because, the troubled interval $(c,0)$ is between this interval and our guessed root. Practically, for many points $x_0\in (0,+\infty)$ the sequence $\{x_n\}$, after some jumping away this way and that, eventually converges to the root and some time it doesn't \cite{khandani}. But, we are not interested in  convergence by chance for our sequence. We notice that, $f$ satisfies all conditions of Theorem $\ref{convergence1}$ on the interval $[a,c)$ for each $a\in (-\infty,-2]$. By Theorem $\ref{convergence1}$  the Newton-Raphson sequence  converge to $c$ for each $x_0\in(-\infty,-2]$. The first 20 terms of $\{x_n\}$ with $x_0=-400$ which is very far from $c$ are as follows:\\

-400, \\
-266.6677819490915, -177.78019734972494, -118.52265267881265, -79.01889929291954, -52.68499811014733, -35.13201021893894, -23.43453810797174, -15.643229227593249, -10.46004014373921, -7.022240831542441, -4.759356916742918, -3.299443626313881, -2.4083528310926825, -1.9439433123997434, -1.7877742400378036, -1.7695329436681617, -1.7692923957961018, -1.7692923542386327, -1.7692923542386314.
\end{example}

{Acknowledgement}
This research has been done and supported financially by Islamic Azad university of Mahabad with project code 235468168.
\bibliographystyle{amsplain}

\begin{thebibliography}{99}
\bibitem{kendal}K. E. Atkinson, An introduction to numerical analysis, second edition, John Wiley $\&$ Sons, 1989.
\bibitem{khandani} H. Khandani, F. Khojasteh, An iterative method for estimation the roots of real-valued functions, preprint:	arXiv:2111.14460, 2021.
\bibitem{sulimayer}  E. S$\ddot u$li, D. F. Mayers, An introduction to numerical analysis, Cambridge university press, 2003.
\bibitem{tj}T. J. Ypma, Historical development of the Newton–Raphson Method, doi:10.1137/1037125, SIAM Review 37 (4), 531–551, 1995.
\end{thebibliography}

\end{document}